\documentclass{amsart}

\usepackage{mathrsfs, mathtools, amssymb}

\usepackage{dsfont}

\usepackage[paper=a4paper, margin=3cm]{geometry}

\usepackage[breaklinks]{hyperref}       
\usepackage{url}            
\usepackage{booktabs}       
\usepackage{amsfonts}       
\usepackage{nicefrac}       
\usepackage{microtype}      
\usepackage{lipsum}

\usepackage{amsmath,amssymb,amsthm,amscd}
\usepackage{amsrefs}
\usepackage{tikz,tikz-cd}
\usetikzlibrary{arrows,arrows.meta,calc}

\usepackage{graphicx}
\graphicspath{}
\DeclareGraphicsExtensions{.pdf,.png,.jpg,.jpeg}

\usepackage{yhmath}
\usepackage{enumitem}
\usepackage{breakurl}
\usepackage{subcaption}

\usepackage{tikz}
\usepackage{tikz-cd}
\usetikzlibrary{arrows}
\usepackage{epsfig}
\usepackage[all]{xy}
\usepackage{epstopdf}
\usepackage{framed}

\newcommand{\Z}{\mathbb Z}
\newcommand{\Q}{\mathbb Q}
\newcommand{\R}{\mathbb R}
\newcommand{\C}{\mathbb C}

\newcommand{\N}{\mathbb N}

\newcommand{\lb}{\lbrace}
\newcommand{\rb}{\rbrace}
\newcommand{\la}{\langle}
\newcommand{\ra}{\rangle}
\renewcommand{\phi}{\varphi}
\newcommand{\eps}{\varepsilon}

\DeclareMathOperator{\rk}{rk}

\DeclareMathOperator{\Id}{Id}

\DeclareMathOperator{\Img}{Im}
\DeclareMathOperator{\GL}{GL}

\DeclareMathOperator{\Homs}{\mathscr{H}\text{\kern -3pt {\calligra\large om}}\,}

\DeclareMathOperator{\str}{star}

\DeclareMathOperator{\GKM}{GKM}

\renewcommand{\leq}{\leqslant}
\renewcommand{\geq}{\geqslant}

\theoremstyle{plain}
\newtheorem{thm}{Theorem}[section]
\newtheorem{lm}[thm]{Lemma}
\newtheorem{cor}[thm]{Corollary}
\newtheorem{pr}[thm]{Proposition}

\theoremstyle{definition}

\newtheorem{rem}[thm]{Remark}

\newtheorem{defn}[thm]{Definition}

\newcounter{thmMaincounter}
\newtheorem{theoremM}[thmMaincounter]{Theorem}

\newtheorem{probM}[thmMaincounter]{Problem}

\tikzcdset{arrow style=tikz, diagrams={>=stealth}}

\begin{document}
    \title{
    Extensions of realizable Hamiltonian and complexity one GKM$_4$ graphs
    }

	\author{Oliver Goertsches}
	\address[O.\,Goertsches]{Philipps-Universit\"at Marburg, Germany}
	\email{goertsch@mathematik.uni-marburg.de}
	
	\author{Grigory Solomadin}
	\address[G.\,Solomadin]{Philipps-Universit\"at Marburg, Germany}
	\email{grigory.solomadin@gmail.com}
	
	\begin{abstract} We prove that the GKM graphs of GKM$_4$ manifolds that are either Hamiltonian or of complexity one extend to torus graphs. The arguments are based on a reformulation of the extension problem in terms of a natural representation of the fundamental group of the GKM graph, using a coordinate-free version of the axial function group of Kuroki, as well as on covers of GKM graphs and acyclicity results for orbit spaces of GKM manifolds.
	\end{abstract}
	
	
	
	\maketitle
	
	\section{Introduction}
	GKM manifolds are a natural, far-reaching generalization of (quasi)toric manifolds, named after Goresky--Kottwitz--MacPherson~\cite{gkm-98}. These are connected, compact, orientable manifolds with vanishing odd-degree cohomology, acted on by a compact torus, in a way that the orbit space of its equivariant one-skeleton has the structure of a graph. This graph, equipped with a labelling given by the weights of the isotropy representation at the fixed points, is called the GKM graph of the action. The condition on the shape of the one-skeleton is equivalent to the weights of any isotropy representation being pairwise linearly independent; in general, if any $k$ of these weights are linearly independent, one calls the action GKM$_k$. 
	The starting point for this note is the following problem.
	
	\begin{probM}[Masuda]\label{prob:mas}
		\begin{enumerate}[label=(\roman*), wide, labelwidth=!, labelindent=0pt, itemindent=!]
			\item Does any $\GKM_{4}$ manifold $M^{2n}$ extend to a torus manifold? (I.e., does the action extend to an effective action of $T^n$?)
			\item Does the GKM graph $\Gamma$ of any $\GKM_{4}$ manifold $M^{2n}$ extend to a torus graph? (I.e., is there an abstract GKM graph with effective $T^n$-labeling which restricts to $\Gamma$?) 
		\end{enumerate}
	\end{probM}
	
	A positive answer to item $(i)$ of Problem~\ref{prob:mas} implies a positive answer to item $(ii)$; in this note we will focus on item $(ii)$. In the language of abstract GKM graphs~\cite{gu-za-01}, we show (see Theorem~\ref{thm:extmain}):
	\begin{theoremM}\label{thm:2}
		Let $\Gamma$ be any unsigned (signed, respectively) $\GKM_{3}$ graph such that the conjugated $2$-faces in $\Gamma$ generate the fundamental group $\pi_{1}(\Gamma)$.
		Suppose that the monodromy along any $2$-face of $\Gamma$ is trivial on the transversal edges of this $2$-face. 
		Then $\Gamma$ admits an extension to an unsigned (signed, respectively) torus graph. 
	\end{theoremM}
	Here, a conjugated $2$-face is any loop given by a $2$-face (aka connection path) in $\Gamma$, conjugated by a path from the chosen base point, in order to yield an element in the fundamental group $\pi_1(\Gamma)$ -- see Definition~\ref{defn:conjugated2face}.
	The condition on the monodromy is automatically satisfied in the GKM$_4$ case. 
	
	Our main tool for the proof of Theorem~\ref{thm:2} is Kuroki's axial function group~\cite{ku-19} and his criterion for the existence of nontrivial extensions of labelings of GKM graphs. The group of axial functions was defined in~\cite{ku-19} for signed GKM graphs (e.g., those of GKM manifolds with an invariant almost complex structure).
	A signed GKM graph yields an unsigned GKM graph by composing the axial function with the quotient of the lattice by $\lb\pm 1\rb$.
 	In Section~\ref{sec:extension} we extend Kuroki's results to unsigned GKM graphs, and use index-free notation to rewrite the axial function group as the group of invariants under a natural representation of the fundamental group $\pi_1(\Gamma)$.
	
	As concrete situations where Theorem~\ref{thm:2} applies we consider two geometrically very different cases. Firstly, we prove (see Corollary~\ref{cor:hamcase}):
	\begin{theoremM} \label{thm:mainthmham}
		Let $\Gamma$ be the GKM graph of any Hamiltonian GKM$_3$ manifold. Then the conjugated $2$-faces in $\Gamma$ generate the fundamental group $\pi_1(\Gamma)$.
	\end{theoremM}
	Combining this statement with Theorem~\ref{thm:2}, it follows that the GKM graph of any Hamiltonian GKM$_4$ manifold extends to a torus graph.  This result positively resolves a particular case of Problem~\ref{prob:mas} $(ii)$, under the additional Hamiltonian assumption.
	Theorem~\ref{thm:mainthmham} follows from a more general statement about GKM$_3$ graphs with an acyclic orientation of the edges, so that every $2$-face has a unique local maximum with respect to the orientation, see Theorem~\ref{thm:hamcase}.
	
	The second situation concerns GKM$_4$ manifolds of complexity $1$, meaning that we are given an effective action of a torus of dimension $n-1$ on a manifold of dimension $2n$.
	
	\begin{theoremM}\label{thm:3}
		The GKM graph of any $\GKM_{4}$ manifold of complexity one admits an extension to a torus graph.
	\end{theoremM}
	
	Theorem~\ref{thm:3} positively resolves another particular case for Problem~\ref{prob:mas} $(ii)$, namely, in complexity $1$.
	To prove Theorem~\ref{thm:3} (see Theorem~\ref{thm:comp1}) we construct a certain GKM cover $p\colon\widetilde{\Gamma}\to\Gamma$ by a (possibly infinite) GKM graph $\widetilde{\Gamma}$ whose conjugated $2$-faces generate the respective fundamental group.
	Then we apply Theorem~\ref{thm:2} to obtain an extension of $\widetilde{\Gamma}$ to a torus graph. We conclude by showing that the action of the deck transformation group on the axial function group of $\widetilde{\Gamma}$ is trivial, making use of the fact that it is a perfect group by the acyclicity results for orbit spaces of equivariant skeleta in~\cite{ay-ma-23}. Notice that Theorem~\ref{thm:3} does not hold for abstract GKM graphs by \cite{so-23}.
	\medskip
	
	\noindent {\bf Acknowledgements:} The authors gratefully acknowledge funding of the Deutsche Forschungsgemeinschaft
	(DFG, German Research Foundation): Project numbers 452427095, and 561158824 (the second named author's Walter Benjamin Fellowship). The second named author is grateful to S.\ Kuroki and M.\ Masuda (in particular for organizing the conference ``Toric Topology in Okayama 2019'', where he was happy to participate and present some of the ideas from the current note for the first time). We wish to thank P.\ Konstantis, L.\ Zoller and N.\ Wardenski for various stimulating discussions on the subject of this note. 
	\section{Graphs, connections, invariant functions}\label{sec:gkmbasic}
	
	This section contains introductory definitions and basic facts about unsigned and signed GKM graphs required for understanding the main results of this note. 
	
	\begin{defn}
		Let $(V,E)$ be an (abstract) connected graph (without loops, but possibly with multiple edges) with vertices $V$, and directed edges $E$ connecting them (such that for any $e\in E$ also the reversed edge $\overline{e}$ is an element of $E$). 
		The initial and terminal vertices of an edge $e\in E$ are denoted by $i(e),t(e)\in V$, respectively. 
		By definition, the \textit{star} of $(V,E)$ at $v\in V$ consists of all edges at the vertex $v$ in $(V,E)$:
		\[
		\str v=\str_{(V,E)} v:=\lb e\in E\mid i(e)=v\rb.
		\]
		The graph $(V,E)$ is called \textit{$n$-valent}, if $|\str v|=n$ holds for any $v\in V$, where $|X|$ denotes the cardinality of a finite set $X$.
	\end{defn}
	
	In what follows, we assume that $(V,E)$ is a connected $n$-valent graph.
	
	\begin{defn}[\cite{gu-za-01}]
		A set $\nabla=\lb \nabla_{e}\mid e\in E\rb$ of bijective maps $\nabla_{e}:\str i(e)\to \str t(e)$ satisfying
		\begin{enumerate}
			\item $\nabla_{\overline{e}}=(\nabla_{e})^{-1}$;
			\item $\nabla_{e}e=\overline{e}$;
		\end{enumerate}
		is called a \textit{(combinatorial) connection} on the graph $(V,E)$.
	\end{defn}
	
	We call any of the two integer vectors $a$ and $-a$ in $\Z^k$ a \textit{lift} of an element $\pm a\in \Z^k/\pm 1$.
	
	\begin{defn}[\cite{gu-za-01}, \cite{go-ko-zo-22}]\label{defn:abstractunsignedgraph}
		An \emph{(abstract) (unsigned) GKM graph} is a tuple $\Gamma=(V,E,\nabla,\alpha)$ consisting of an abstract graph $(V,E)$ as above, a connection $\nabla$ on $(V,E)$, as well as an \emph{(unsigned) axial function} $\alpha:E\to \Z^k/\pm 1$ that satisfy:
		\begin{enumerate}
			\item For each $e\in E$, $\alpha(\overline{e}) = \alpha(e)$.
			\item For each $v\in V$ and $e,e'\in \str v$ with $e\neq e'$, $\alpha(e)$ and $\alpha(e')$ are linearly independent.
			\item The axial function $\alpha$ is \emph{compatible with the connection}, i.e., for all edges $e,e'$ at any vertex $v\in V$ and for any lift $\widetilde{\alpha}(\nabla_e e')$ of $\alpha(\nabla_e e')$ and $\widetilde{\alpha}(e')$ of $\alpha(e')$, there exists a sign $\varepsilon\in \{\pm 1\}$ such that 
			\[
			\widetilde{\alpha}(\nabla_e e') \in \varepsilon  \widetilde{\alpha}(e') + \Z \alpha(e).
			\]
		\end{enumerate}
	\end{defn}
	\begin{rem}
		Often, in the definition of an abstract GKM graph, the connection is not assumed to be part of the structure; instead, it is only assumed that there exists a connection that is compatible with a given axial function. See for instance~\cite{gu-za-01}. For the results of this note it is natural to fix also the connection, as we are concerned with extension problems with respect to a given connection, see Definition \ref{def:ext} below.
	\end{rem}
	
	We will need to keep track of the signs $\varepsilon$ and the integer coefficients in the above congruence relations. To this end we fix, for easier bookkeeping, an arbitrary global lift
	\[
	\widetilde{\alpha}:E\longrightarrow \Z^k
	\]
	of $\alpha$ (i.e., $\pi\circ \widetilde{\alpha} = \alpha$, where $\pi:\Z^k\to \Z^k/\pm 1$ is the canonical projection), satisfying no further conditions except $\widetilde{\alpha}(\overline{e})= -\widetilde{\alpha}(e)$ for all $e\in E$. This lift $\widetilde{\alpha}$ is not to be confused with the choice of a signed structure on the GKM graph (see Definition \ref{defn:abstractsignedgraph} below).
	
	In order to avoid having to choose an ordering of the edges at each vertex, we introduce the following notation.
	Define a nondegenerate scalar product on the free $\Z$-module $\Z\str v$ generated by the finite set $\str v$ by requiring $\str v$ to be an orthonormal basis, i.e.,
	\[
	\la e,e'\ra=\delta_{e}^{e'},
	\]
	where $\delta$ is the Kronecker delta.
	The connection $\nabla$ induces homomorphisms 
	\[
	\nabla_{e}:\ \Z \str i(e)\to \Z \str t(e),
	\] 
	of $\Z$-modules by acting on the generators.
	
	Denote the group $\lb \pm 1\rb$ by $\mu_{2}$. 
	Let $\mu_2 \str v$ be the free $\mu_{2}$-module generated by the finite set $\str v$.
	This group acts on $\Z \str v$ by the formula below:
	\[
	\mu_2 \str v\times \Z \str v\to \Z \str v,\ 
	x\cdot y := \sum_{e\in \str v} \langle x,e\rangle \langle y,e\rangle e,
	\]
	which is nothing but componentwise multiplication in the standard basis. Here, we wrote $x\in \mu_2 \str v$ in the standard basis as $x = \sum_{e\in \str v} \langle x,e\rangle e$, with $\langle x,e\rangle \in \mu_2$.
    In a similar way to the above case of $\Z$-modules, we extend $\nabla_{e}$ uniquely to a homomorphism of free $\mu_{2}$-modules $\mu_{2}\str i(e)\to \mu_{2}\str t(e)$.
	
	Given the lift $\widetilde\alpha$ of the axial function, we denote, for $e\in E$, by $\varepsilon(e)\in \mu_2 \str i(e)$ and by $c(e)\in \Z \str i(e)$ the signs and integers such that
	\begin{equation} \label{eq:epsandc}
		\widetilde{\alpha}(\nabla_e e') = \langle \varepsilon(e),e'\rangle\cdot \widetilde{\alpha}(e') + \langle c(e),e'\rangle \cdot \widetilde{\alpha}(e);
	\end{equation}
	we make $\varepsilon$ and $c$ unique by requiring 
	\begin{equation}\label{eq:epsandcunique}
		\langle \varepsilon(e),e\rangle = 1 \quad \textrm{and} \quad \langle c(e),e\rangle = -2.
	\end{equation}
	
	\begin{defn}\label{defn:invfunc} Consider a graph $(V,E)$, together with a connection $\nabla$. A \textit{sign function} $\eps$ on $(V,E,\nabla)$ is a collection $\eps=(\eps(e))_{e\in E}$ of elements $\eps(e)\in\mu_{2} \str i(e)$ such that
		\begin{enumerate}
			\item $\la \eps(e),e\ra=1$;
			\item $\nabla_{e} \eps(e)=\eps(\overline{e})$.
		\end{enumerate}
		hold for all $e\in E$. Given a sign function $\varepsilon$, an \textit{invariant function} $c$ on $(V,E,\nabla,\eps)$ is a collection $c=(c(e))_{e\in E}$ of elements $c(e)\in\Z \str i(e)$ such that
		\begin{enumerate}[label=(\arabic*$'$)]
			\item $\la c(e),e\ra=-2$;
			\item $\nabla_{e} c(e)=\eps(\overline{e})\cdot c(\overline{e})$.
		\end{enumerate}
		hold for all $e\in E$.
	\end{defn}
	
	The notion of an invariant function was introduced, in the context of signed GKM graphs, in~\cite{ku-19}.
	\begin{lm}
		Given an abstract unsigned GKM graph, $\varepsilon = (\varepsilon(e))_{e\in E}$ and $c = (c(e))_{e\in E}$ as defined by~\eqref{eq:epsandc} and~\eqref{eq:epsandcunique} are a sign respectively invariant function.
	\end{lm}
	\begin{proof}
		The two conditions $(1)$ and $(1')$ are satisfied by definition. To prove the other two, we use~\eqref{eq:epsandc} as follows:
		\[
		\widetilde{\alpha}(e') = \widetilde{\alpha}(\nabla_{\overline{e}}\nabla_e e') = \langle \varepsilon(\overline{e}),\nabla_e e'\rangle \widetilde{\alpha}(\nabla_e e') + \langle c(\overline{e}),\nabla_ee'\rangle \widetilde{\alpha}(\overline{e}),
		\]
		which, comparing with~\eqref{eq:epsandc}, gives
		\[
		\langle \varepsilon(\overline{e}),\nabla_ee'\rangle =\langle \varepsilon(e),e'\rangle \quad\textrm{and} \quad \langle \varepsilon(\overline{e}),\nabla_e e'\rangle\langle c(\overline{e}),\nabla_e e'\rangle = \langle c(e),e'\rangle.
		\]
		This shows that 
		\[
		\varepsilon(\overline{e}) = \sum_{e'} \langle \varepsilon(\overline{e}),\nabla_e e'\rangle \nabla_e e' = \nabla_e \sum_{e'} \langle \varepsilon(e),e'\rangle e' = \nabla_e \varepsilon(e),
		\]
		i.e., $(2)$, and analogously
		\[
		\varepsilon(\overline{e})\cdot c(\overline{e}) = \sum_{e'}  \langle \varepsilon(\overline{e}),\nabla_e e'\rangle \langle c(\overline{e}),\nabla_e e'\rangle \nabla_ee' = \nabla_e \sum_{e'} \langle c(e),e'\rangle  e' = \nabla_e  c(e),
		\]
		i.e., $(2')$.
	\end{proof}
	
	\begin{defn}[\cite{gu-za-01}]
		\label{defn:abstractsignedgraph}
		An \emph{(abstract) (signed) GKM graph} $\Gamma$ consists of an abstract graph $(V,E)$, a connection $\nabla$ on $(V,E)$, as well as a \emph{(signed) axial function} $\alpha:E\to \Z^k$ that satisfy:
		\begin{enumerate}
			\item For each $e\in E$, $\alpha(\overline{e}) = -\alpha(e)$.
			\item For each $v\in V$ and $e,e'\in \str v$ with $e\neq e'$, $\alpha(e)$ and $\alpha(e')$ are linearly independent.
			\item The signed axial function $\alpha$ is \emph{compatible with the connection}, i.e., for all edges $e,e'$ at any vertex $v\in V$, 
			\[
			\alpha(\nabla_ee') \in \alpha(e') + \Z \alpha(e).
			\]
		\end{enumerate}
	\end{defn}
	In other words, a signed GKM graph is nothing but an unsigned GKM graph for which we were able to choose a lift of the unsigned axial function whose associated sign function $\varepsilon$ satisfies $\langle \varepsilon(e),e'\rangle = 1$ for all edges $e,e'$. The associated invariant function then satisfies $\nabla_e c(e) = c(\overline{e})$. Any proof of a statement for general unsigned graphs below will therefore automatically give a statement and its proof for signed graphs, by considering the special case that $\langle \varepsilon(e),e'\rangle = 1$ for all $e,e'$.
	
	\begin{defn}[\cite{gu-za-01}, \cite{go-ko-zo-22}]\label{def:gkmgr}
		Let $\Gamma$ be a GKM graph. It is called a \emph{GKM$_q$ graph} if its axial function $\alpha$ is \emph{$q$-independent}, i.e., if for any $v\in V$ and any pairwise different $e_{1},\dots,e_{q}\in \str v$ the vectors $\alpha(e_{1}),\dots,\alpha(e_{q})$ are linearly independent.
	\end{defn}
	Note that for a GKM$_3$ graph the connection $\nabla$ is determined uniquely by the graph and its axial function.
	
	\begin{defn}
		Consider a signed or unsigned GKM graph $\Gamma$, with axial function $\alpha:E\to \Z^k$ respectively $\alpha:E\to \Z^k/\pm 1$. The axial function $\alpha$, respectively the GKM graph $\Gamma$, is called \emph{effective} if the $\Z$-span of the image of $\alpha$ is all of $\Z^k$. It is called \emph{almost effective} if the $\Z$-span of the image of $\alpha$ is of finite index in $\Z^k$, i.e., a rank $k$ sublattice in $\Z^k$.
		The GKM graph $\Gamma$ is called an \textit{$(n,k)$-type GKM graph} if $\Gamma$ is $n$-valent and the axial function $\alpha:E\to \Z^k $ respectively $\alpha:E\to \Z^k/\pm 1$ is almost effective.
		An $(n,n)$-type (signed or unsigned, respectively) GKM graph is called a \textit{(signed or unsigned, respectively) torus graph}.
	\end{defn}
	
	\begin{defn}[\cite{gu-za-01}, \cite{go-ko-zo-22}]\label{defn:gkm}
		Let $M^{2n}$ be a smooth, closed, connected, orientable manifold with a smooth effective action of the compact torus $T=T^k=(S^1)^k$. 
		This action is called a \textit{GKM action} and $M^{2n}$ is called a \textit{GKM manifold}, if
		\begin{enumerate}[label=(\roman*)]
			\item The fixed point set $M^{T}$ is finite and non-empty;
			
			\item For any point $x\in M^{T}$ the weights of the isotropy representation of $T$ on $T_{x} M$ are pairwise linearly independent;
			
			\item (\textit{Equivariant formality}) One has $H^{odd}(M;\Z)=0$ for the respective singular cohomology groups. 
		\end{enumerate}
		In case $n=k$, a manifold $M$ satisfying items $(i)$ and $(ii)$ is called a \textit{torus manifold}. 
	\end{defn}
	
	\begin{rem}
		The proof of Theorem~\ref{thm:3} (see below) requires the assumption of equivariant formality over $\Z$.
	\end{rem}
	
	As is well known, to a $T^k$-GKM manifold $M$ one may associate an abstract unsigned GKM graph, as follows: the assumptions on a GKM manifold imply that the one-skeleton 
	\[
	M_1:=\{p\in M\mid \dim T\cdot p\leq 1\}
	\]
	is a finite union of $T$-invariant $2$-spheres; one considers the graph $\Gamma$ whose vertices are the fixed points and whose edges correspond to these invariant $2$-spheres, with labels given by the corresponding weights of the isotropy representations.  (After having fixed an isomorphism between the weight lattice in ${\mathfrak{t}}^*$ and $\Z^k$). In case the GKM manifold $M$ admits an invariant almost complex structure, the weights of the isotropy representations are well-defined elements in $\Z^k$, not only in $\Z^k/\pm 1$. 
	
	The fact that this labelled graph admits a compatible connection can be found e.g. in~\cite{gu-za-01}; also see~\cite[Prop.\ 2.3]{go-wi-22}.
	
	\section{Faces and orbit spaces of GKM manifolds}\label{sec:lattices}
	
	This section contains several known facts on faces of GKM manifolds, GKM graphs and orbit spaces, then proceeds with the proof of Theorem~\ref{thm:mainthmham}.
	
	\subsection{Definitions and basic facts}
	
	\begin{defn}[\cite{ay-ma-23}]\label{def:facesubm}
		For any $T$-GKM manifold $M$ and any $i=0,\ldots,\dim T$ let $M_{i}=\lb x\in M\mid \dim Tx\leq i\rb$ be the \textit{equivariant $i$-skeleton} of the $T$ manifold $M$.
		For the natural projection $p\colon M\to Q:=M/T$ let $Q_{i}:=p(M_{i})$.
		A \textit{face of the orbit space} is the closure $F$ of any connected component of $Q_{i}\setminus Q_{i-1}$ such that $F$ contains at least one vertex (i.e. preimage of a fixed point in $M$) of $Q$.
		Let $\rk F:=i$ be the \textit{rank}  of $F$.
		A \textit{face submanifold} $M_{F}\subset M_{i}$, $i=\rk F$, of a GKM manifold $M$ is defined to be the preimage $p^{-1}(F)$ for a face $F$ of $Q$.
	\end{defn}
	
	\begin{defn}[\cite{gu-za-01}, \cite{ay-ma-so-23}]\label{defn:gfac}
		A \textit{face $\Theta$ of a GKM graph $\Gamma$} is a connected subgraph whose edges are invariant with respect to the connection of $\Gamma$ along any edge path in $\Theta$. A face of a GKM graph is necessarily a $k$-valent subgraph, for some $k$; we speak about a \emph{$k$-face} of $\Gamma$. An $(n-1)$-face $\Theta$ of an $n$-valent GKM graph $\Gamma$ is called a \textit{facet} of $\Gamma$.
	\end{defn}
	
	In what follows by a face we mean either a face submanifold, face of an orbit space or a face of a GKM graph, whenever it is clear from the context.
	By the following lemma, the condition on the existence of a fixed point (i.e.\ of a vertex of a face) in Definition~\ref{def:facesubm} is superfluous for GKM manifolds.
	(The claim of this lemma is a particular case of~\cite[Lemma 2.2]{ma-pa-06} in the case of GKM manifolds, specifically, having only finitely many $T$-fixed points.)
	
	\begin{lm}[{\cite[Lemma 2.2]{ma-pa-06}}]\label{lm:efface}
		Any closure $X$ of any connected component of the set $M_{i}\setminus M_{i-1}$ of a GKM $T$-manifold $M$ is a smooth $T$-invariant submanifold, where $i\leq \rk T$. The $T$-action on $X$ is again equivariantly formal, with finite nonempty fixed point set (see Definition~\ref{defn:gkm}). 
	\end{lm}
	
	The following claim is straightforward to prove.
	\begin{lm}
		Let $\Gamma$ be any $\GKM_{3}$ graph.
		Then any two edges with a common vertex belong to a $2$-face of $\Gamma$.
	\end{lm}
	
	Recall that a topological space $X$ is called \textit{$n$-acyclic} if $\widetilde{H}^{i}(X;\Z)=0$ holds for any $i\leq n$.
	
	\begin{thm}[{\cite[Thm. 2]{ay-ma-23}, \cite[Prop. 3.11 (4)]{ay-ma-so-23}}]\label{thm:ams}
		For a $\GKM_{j}$ manifold $M$ with orbit space $Q$, the space $Q_{i}$ is $\min\lb i-1, j+1\rb$-acyclic for any natural number~$i$. 
	\end{thm}
	
	The following corollary will be used in the proof of Theorem~\ref{thm:comp1} below.
	
	\begin{cor}\label{cor:perfg}
		For any $\GKM_{3}$ manifold the group $\pi_{1}(Q_{2})$ is perfect. 
	\end{cor}
	\begin{proof}
		Indeed, by Theorem~\ref{thm:ams}, the abelianization $H_{1}(Q_{2})$ of $\pi_{1}(Q_{2})$ is trivial.
	\end{proof}
	
	\begin{defn}\label{defn:conjugated2face}
		Let $\Gamma$ be a GKM$_3$ graph, and $v$ any fixed vertex. We call the class $g^{-1}g' g$ in $\pi_{1}(\Gamma,v )$ a \textit{conjugated $2$-face} of $\Gamma$, where $g$ is any edge path in $\Gamma$ with $i(g)=v$ and $g'$ is any loop with $i(g')=t(g)$ following the edge path of a $2$-face in $\Gamma$.
	\end{defn}
	
	\subsection{Proof of Theorem~\ref{thm:mainthmham}}
	
	\begin{defn}
		For any graph $(V,E)$ with an \textit{orientation of the edges}, i.e. any function $o\colon E\to \lb \pm 1\rb$ with $o(\overline{e})=-o(e)$, the number of negative values of $o$ on $\str v$ is called the \textit{index of a vertex} $v$ in $(V,E)$ with respect to $o$.
		The $2i$-th \textit{Betti number $b_{2i}((V,E),o)$ of the graph $(V,E)$ with orientation $o$ of edges} is defined as the number of vertices in $(V,E)$ of index $i$ with respect to $o$.
		We will call any edge loop in $(V,E)$ a \emph{cycle}. A cycle consisting of only positively or only negatively oriented edges with respect to $o$ is called an \textit{$o$-oriented cycle}. 
		An orientation $o$ of edges in the graph $(V,E)$ is called \textit{acyclic} (this condition is called the \textit{no-cycle condition} in~\cite{gu-za-01}) if the respective graph admits no $o$-oriented cycles. 
	\end{defn}
	
	\begin{rem}
		Suppose $\Gamma$ is a signed GKM graph, with axial function $\alpha\colon E\to\Z^{k}$.
		For a generic vector $\xi\in(\Z^{k})^{*}$ the pairing $\la\alpha(e),\xi\ra$ is nonzero and induces an orientation $o=o_{\xi}$ \cite[\S 1.4]{gu-za-01}.
		The respective numbers $b_{2i}(\Gamma,o_{\xi})$ are called combinatorial Betti numbers and denoted by $b_{2i}(\Gamma)$ in \cite[\S 1.4]{gu-za-01}.
		As the notation suggests, $b_{2i}(\Gamma)$ does not depend on the choice of $\xi\in(\Z^{k})^{*}$.
		Furthermore, in the Hamiltonian setting $\Gamma=\Gamma(M)$, one has $b_{2i}(\Gamma)=b_{2i}(M)$ by \cite[Formula (0.4)]{gu-za-01}, which explains the nomenclature.
		For an arbitrary (acyclic) orientation $o$ on edges of $\Gamma$, $b_{2i}(\Gamma,o_{\xi})$ depends on the choice of $o$, in general (see Fig.~\ref{fig:difbn}), and such numbers do not relate to topological Betti numbers.
	\end{rem}
	
	\begin{figure}[h]
		\begin{center}
			\begin{tikzpicture}[node distance=3cm, thick,main node/.style={circle,draw,font=\sffamily\Large\bfseries}]
				\fill (0,0) coordinate (a) circle (2pt);
				\fill (1,0) coordinate (b) circle (2pt);
				\fill (1,1) coordinate (c) circle (2pt);
				\fill (0,1) coordinate (d) circle (2pt);
				\draw[->,>=stealth'] (a)--(b);
				\draw[->,>=stealth'] (b)--(c);
				\draw[->,>=stealth'] (d)--(c);
				\draw[->,>=stealth'] (a)--(d);
			\end{tikzpicture}\qquad
			\begin{tikzpicture}[node distance=3cm, thick,main node/.style={circle,draw,font=\sffamily\Large\bfseries}]
				\fill (0,0) coordinate (a) circle (2pt);
				\fill (1,0) coordinate (b) circle (2pt);
				\fill (1,1) coordinate (c) circle (2pt);
				\fill (0,1) coordinate (d) circle (2pt);
				\draw[->,>=stealth'] (a)--(b);
				\draw[->,>=stealth'] (c)--(b);
				\draw[->,>=stealth'] (c)--(d);
				\draw[->,>=stealth'] (a)--(d);
			\end{tikzpicture}
			\caption{The same GKM graph with two different acyclic edge orientations providing different fourth Betti numbers (i.e. numbers of vertices that are local maxima with respect to a given orientation)}
			\label{fig:difbn}
		\end{center}
	\end{figure}

	\begin{thm}\label{thm:hamcase}
		Let $\Gamma$ be any unsigned $\GKM_{3}$ graph with an acyclic orientation $o\colon E\to \lb \pm 1\rb$ of edges such that $b_{4}(F,o)=1$ holds for any $2$-face $F$ in $\Gamma$.
		Then the conjugated $2$-faces in $\Gamma$ generate the fundamental group $\pi_{1}(\Gamma)$.
	\end{thm}
	\begin{proof}
		The condition that the graph has no oriented cycles with respect to the chosen orientation of the edges implies that the orientation $o$ induces a partial order $\leq_{o}$ on the vertices of $\Gamma$, i.e., $u\leq_{o} v$ if and only if $u=v$ or there is an edge path $e_1,\ldots,e_k$ connecting $u$ and $v$ with $o(e_i)=1$ for all $i$.
		
		Choose any total order extending this partial order.
		Using this extension we can define an injective function $f\colon V\to \N$ on the vertices of $\Gamma$ by consecutive integers from $1$ to $N$ satisfying that for every edge $e$ with $o(e)=1$, i.e., $i(e)<_{o} t(e)$, we have $f(i(e))< f(t(e))$. 
		
		By the assumption on the Betti numbers of the $2$-faces, any $2$-face $F$ of $\Gamma$ has a unique locally maximal vertex with respect to the induced order on $F$.
		We say that a cycle $g$ in $\Gamma$ has \textit{height $h$ and multiplicity $m$} if the maximum of $f$ on $V(g)$ equals $h$ and is attained $m$ times on a vertex of $g$.
		
		Given any cycle $g$ at the least vertex $v_0$ of $\Gamma$, we wish to decompose $g$ into a product of conjugated $2$-faces.
		This is done by double induction on height and multiplicity of $g$.
		The double induction is done with respect to the lexicographical order $\leq_{lex}$ on $(h,m)\in\N\times\N$.
		The base of the induction is $(h,m)=(1,1)$.
		In this case, the edge path $g$ is constant and there is nothing to prove.
		Assume that the claim holds for $\leq_{lex}(h-1,m)$ and $\leq_{lex}(h,m-1)$.
		Let $g$ be an edge path with height and multiplicity $(h,m)$.
		Let $v$ be the maximal vertex of $g$, and consider one of its occurrences, with adjacent edges $uv$, $vw$.
			
		Let $F$ be the $2$-face spanned by $uv$ and $vw$. The uniqueness of a local maximum on $F$ implies that any vertex $z \neq v$ of $F$ satisfies $z<_ov$ ($v$ is the maximal vertex of $F$).
		For $h_2 :=  vw\cdot uv$ choose $f_2$ so that $f_{2}^{-1}h_2$ is the simple loop at $u$ bounding $F$. Note that in case $u=w$ the face $F$ is a biangle, and the path $f_2$ is constant.
		Then we can decompose $g = h_3 h_2 h_1$ for certain paths $h_1$ and $h_3$, and write
		\[
		g = g_2 g_1 ,\quad \textrm{where}\quad       g_1:=h_1^{-1} f_2^{-1} h_2 h_1 \textrm{ and } g_2:=h_3 f_2 h_1.
		\]
		Notice that $g_1$ is a conjugated $2$-face corresponding to $F$, and $g_2$ is a path with lower multiplicity or height, as it traverses $F$ from the other side avoiding $v$, see Fig.~\ref{fig:indmor}. 
		Therefore, $g_{2}$ decomposes by the induction assumption.
		Hence, $g$ decomposes as well.
		This completes the induction step, and the proof is complete.
		
		\begin{figure}[h]
			\begin{center}
				\begin{tikzpicture}[node distance=3cm, thick,main node/.style={circle,draw,font=\sffamily\Large\bfseries}]
					\fill (0,1) coordinate (v) circle (2pt);
					\fill (-1,0) coordinate (u) circle (2pt);
					\fill (1,0) coordinate (w) circle (2pt);
					\fill (-1,-1) coordinate (z) circle (2pt);
					\fill (0,-2) coordinate (v0) circle (2pt);
					\fill (1,-1) coordinate (zr);
					\fill (-2,-1) coordinate (zl);
					\fill (2,-1) coordinate (zrr);
					\fill (-1.5,0.5) coordinate (vl);
					\fill (0,1) coordinate (vr);
					\node[above] at (0,1) {$v$};
					\node[right] at (-1,0) {$u$};
					\node[left] at (1,0) {$w$};
					\node[right] at (-1,-1) {$z$};
					\node[below] at (0,-2) {$v_{0}$};
					\draw[->,>=stealth'] (u)--(v);
					\draw[->,>=stealth'] (v)--(w);
					\draw[dashed] (u) to [out=210,in=150] (z);
					\draw[dashed,->,>=stealth'] (z) to [out=330,in=210] (zr);
					\draw[dashed] (zr) to [out=30,in=330] (w);
					\draw[->,>=stealth'] (v0)
					to [out=180,in=270] (zl);
					\draw (zl)
					to [out=90,in=180] (vl) 
					to [out=0,in=120] (u);
					\draw[->,>=stealth'] (w)
					to [out=60,in=0] (vr)
					to [out=0,in=90] (zrr);
					\draw (zrr) to [out=270,in=0] (v0);
					\node[below] at (-1.5,1) {$g$};
					\node at (0,-1) {$f_{2}$};
					\node[left] at (-2,-0.5) {$h_{1}$};
					\node[right] at (2,-0.5) {$h_{3}$};
					\node at (0,-0.25) {$F$};
					\node at (0,0.5) {$h_{2}$};
				\end{tikzpicture}
				\caption{The induction step replaces the edge path $u,v,w$ with the (dashed) edge path $u,z,w$ in a cycle $g$ with a maximal vertex $v$ of multiplicity $2$. The vertices $u,v,w,z$ belong to the same $2$-face $F$. The resulting cycle has the maximal vertex $v$ with multiplicity $1$}
				\label{fig:indmor}
			\end{center}
		\end{figure}
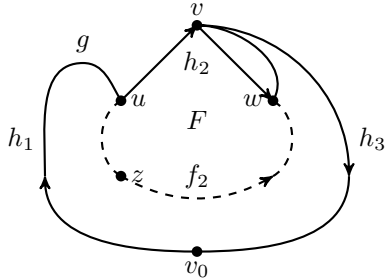
	\end{proof}
	
	\begin{rem}
		An injective function $f\colon V\to \N$ induces a total order on $V$; as used in the proof, any total order on $V$ is induced by such a function.
		An acyclic orientation $o$ of the edges induces the partial order $\leq_o$ on $V$ with the property that the two boundary vertices of any edge are comparable. In the proof above, we used the well-known fact that it is always possible to extend this partial order $\leq_{o}$ to a total order.
		Hence, in Theorem~\ref{thm:hamcase} we could equivalently assume that we are given a total order on $V$ with the property that the induced orientation $o$ on edges is acyclic and satisfies $b_{4}(F,o)=1$ for any $2$-face $F$ in $\Gamma$.
	\end{rem}
	
	In the particular case of a Hamiltonian $\GKM_{3}$ manifold we have the following corollary (Theorem \ref{thm:mainthmham} in the introduction):
	
	\begin{cor}\label{cor:hamcase}
		Let $\Gamma$ be the GKM graph of any Hamiltonian $\GKM_{3}$ manifold $M$.
		Then the conjugated $2$-faces in $\Gamma$ generate the fundamental group $\pi_{1}(\Gamma)$.
	\end{cor}
	\begin{proof} 
		A generic component of a momentum map of the action induces an acyclic orientation on $\Gamma$ by~\cite[Theorem 1.4.2]{gu-za-01}, which satisfies $b_{4}(F)=1$ for any $2$-face $F$ in $\Gamma$ by~\cite[Theorem 1.4.4]{gu-za-01}. Then the claim follows from Theorem~\ref{thm:hamcase}.
	\end{proof}
	
	\begin{rem}
		In this theorem it is sufficient to consider equivariant formality over $\Q$, which is automatic for Hamiltonian torus actions by Kirwan's result \cite[Prop. 5.8]{ki-84}.
	\end{rem}
	
	\begin{rem}
		Corollary~\ref{cor:hamcase} is not true in the GKM$_2$ case, and Theorem~\ref{thm:hamcase} does not hold without the assumption on the Betti numbers of the $2$-faces.
		For example, the standard $T^{2}$-action on the flag manifold $Fl_{3}={\mathrm{SU}}(3)/T^2$ has
		the GKM graph $\Gamma$ as in Figure~\ref{fig:Fl3}. As the action is Hamiltonian, a generic component of the momentum map induces an acyclic orientation on $\Gamma$. Equipped with the canonical connection of a homogeneous GKM manifold~\cite{gu-ho-za-06}, it has a $2$-face $F$ consisting of a $6$-cycle subgraph that satisfies $b_{4}(F)=2$.
		On the other hand, the minimal normal subgroup generated by conjugated $2$-faces of $\Gamma$ has index $2$ in $\pi_{1}(\Gamma)$. This follows for instance because gluing in a $2$-disc to each $2$-face of $\Gamma$ results in $\R P^2$, see~\cite[Example 3.13 (b)]{go-ko-zo-22'}.
		
		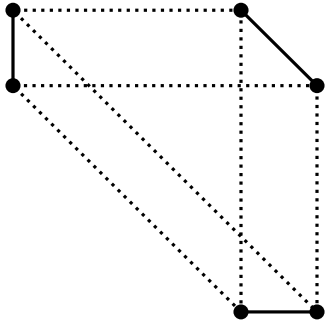
\begin{figure}[h]
			\begin{center}
				\begin{tikzpicture}
					\draw[very thick, dotted] (1,1) -- ++(-3,0) -- ++(4,-4) -- ++(0,3);
					\draw[very thick] (2,0) -- ++(-1,1);
					\draw[very thick] (2,-3)--++(-1,0);
					\draw[very thick, dotted] (1,-3)--++(0,4);
					\draw[very thick, dotted] (1,-3) -- ++(-3,3) -- ++ (4,0);
					\draw[very thick] (-2,0) -- ++(0,1);
					\node at (1,1)[circle,fill,inner sep=2pt]{};
					
					\node at (-2,1)[circle,fill,inner sep=2pt]{};
					
					\node at (1,-3)[circle,fill,inner sep=2pt]{};
					
					\node at (-2,0)[circle,fill,inner sep=2pt]{};
					
					\node at (2,0)[circle,fill,inner sep=2pt]{};
					
					\node at (2,-3)[circle,fill,inner sep=2pt]{};
					
				\end{tikzpicture}
				\caption{GKM graph of the flag manifold $Fl_3={\mathrm{SU}}(3)/T^2$. It has a $2$-face $F$ with $b_{4}(F)=2$, depicted by dotted lines}
				\label{fig:Fl3}
			\end{center}
		\end{figure}
	\end{rem}
	
	\section{Extension for unsigned GKM graphs}\label{sec:extension}
	
	In this section we generalize the group of axial functions, introduced by Kuroki~\cite{ku-19} for signed GKM graphs, to unsigned GKM graphs. We generalize the extension criterion for a signed GKM graph from~\cite{ku-19}, see Theorem~\ref{thm:kurrk} below. Recall from Section~\ref{sec:gkmbasic} that a constant sign function $\eps \equiv 1$ corresponds to the signed case; by allowing arbitrary sign functions we treat the signed and unsigned cases simultaneously.
	
	\subsection{Group of label functions} \label{subsec:labelfct}
	
	\begin{defn}\label{def:cong}
		Let $\Gamma$ be a graph with a connection $\nabla$, equipped with a sign function $\eps$ and an invariant function $c$ (cf.\ Definition \ref{defn:invfunc}). Define $L(\Gamma)$ to be the group of functions $a: E\to\Z$ satisfying 
		\begin{equation}\label{eq:afrel}
			a(\nabla_{e} e')=\la \eps(e),e'\ra\cdot a(e')+\la c(e),e'\ra\cdot a(e)
		\end{equation} 
		for all edges $e,e'$ at any vertex $v\in V$.
		We call $L(\Gamma)$ the \textit{group of label functions} for $\Gamma$.
	\end{defn}
	In particular, putting $e'=e$ in~\eqref{eq:afrel} we obtain, using $\langle \eps(e),e\rangle = 1$ and $\langle c(e),e\rangle = -2$, that
	\begin{equation}\label{eq:gengood}
		a(\overline{e})=-a(e)
	\end{equation}
	holds for all edges $e\in E$. 
	
	Recall that a set of linearly independent elements $v_1,\ldots,v_k$ in a lattice $L$ is called a \emph{primitive set in $L$} if the rational vector space $V$ spanned by $v_1,\ldots, v_k$ in $L\otimes \Q$ satisfies that $V\cap L$ equals the $\Z$-span of $v_1,\ldots,v_k$. 
	It is a well-known fact that a subset of $L$ can be extended to a lattice basis of $L$ if and only if it is primitive. 
	
	\begin{lm}\label{lem:componentsofaxialfunction}
		Let $\Gamma$ be an unsigned GKM graph, with a given lift $\widetilde{\alpha}:E\to \Z^k$ of the axial function $\alpha$, and associated sign and invariant functions $\varepsilon$ and $c$ (see Section~\ref{sec:gkmbasic}). Then the components $\widetilde{\alpha}_i:E\to \Z$ of $\widetilde{\alpha}$, $i=1,\ldots,k$, are elements of $L(\Gamma)$. Furthermore:
		\begin{enumerate}
			\item  
			$\Gamma$ is almost effective if and only if these elements are linearly independent in $L(\Gamma)$.
			\item $\Gamma$ is effective if and only if $\{\widetilde{\alpha}_1,\ldots,\widetilde{\alpha}_k\}$ forms a primitive set in $L(\Gamma)$.
		\end{enumerate}
	\end{lm}
	\begin{proof} Equation~\eqref{eq:epsandc} says that 
		\[
		\widetilde{\alpha}(\nabla_e e') = \langle \varepsilon(e),e'\rangle\cdot \widetilde{\alpha}(e') + \langle c(e),e'\rangle \cdot \widetilde{\alpha}(e),
		\]
		holds for all edges $e,e'$ at any vertex $v\in V$, where $\nabla$ is the connection on $\Gamma$. 
		The first claim then follows directly by considering the $k$ components of this equation.
		
		The condition that the elements $\widetilde{\alpha}_i,\ i=1,\dots,k$, are linearly dependent in $L(\Gamma)$ is equivalent to the existence of integer constants $c_i$ such that $\sum_i c_i \widetilde{\alpha}_i = 0$. But this is equivalent to the condition that $\sum_i c_i \widetilde{\alpha}_i(e) = 0$ holds for each edge $e\in E$, i.e., that there is a linear equation in $\Z^k$ that is satisfied by every element in the image of $\widetilde{\alpha}$. In turn, this is the same as the axial function $\alpha$ not being almost effective. 
		
		Assuming that the collection $\widetilde{\alpha}_i$, $i=1,\dots,k$, does not form a primitive set in $L(\Gamma)$, we find a rational $k\times k$-matrix $C=(c_{ij})$ whose inverse $C^{-1}$ has only integer entries and determinant $|\det C^{-1}| \neq \pm 1$, such that the elements $\beta_i:= \sum_{j=1}^k c_{ij} \widetilde{\alpha}_j$, $i=1,\ldots,k$, form a primitive set in $L(\Gamma)$. But then $\beta:=(\beta_1,\ldots,\beta_k)\colon E\to \Z^k$ is an axial function as well; this shows that the $\Z$-span of the image of $\alpha$ is a proper sublattice of $\Z^k$, i.e., $\alpha$ is not effective. Conversely, we argue in the same way: if $\alpha$ is not effective, we find a rational matrix $C$ of the same type which sends an integer basis of the $\Z$-span of ${\mathrm{im}}\, \alpha$ to the standard basis of $\Z^k$. Then, putting $\beta_i:=\sum_{j=1}^k c_{ij}\widetilde{\alpha}_j\in L(\Gamma)$, the $\Z$-span of the $\widetilde{\alpha}_i$ is a proper sublattice of the $\Z$-span of the $\beta_i$. Consequently, the $\widetilde{\alpha}_i$ cannot form a primitive set in $L(\Gamma)$.
	\end{proof}
	\begin{rem}\label{rem:labax}
		Conversely, any collection $a_{1},\dots,a_{k}\in L(\Gamma)$ of label functions defines a map
		$(a_1,\ldots,a_k)\colon E\to \Z^k$ satisfying the congruence relation~\eqref{eq:epsandc}, and hence condition (3) on the axial function in Definition~\ref{defn:abstractunsignedgraph} of an abstract unsigned GKM graph (respectively Definition~\ref{defn:abstractsignedgraph} of an abstract signed GKM graph in case $\eps\equiv 1$).
	\end{rem}
	
	\subsection{Group of axial functions}
	
	The idea in~\cite{ku-19} is to rewrite the label functions as maps on the vertices of $\Gamma$ instead of the edges. We extend the definition of axial function groups to the case of an unsigned GKM graph $\Gamma$.
	Using this idea, as well as coordinate-free notation, we show that these groups are isomorphic to the invariants of a suitable representation of the fundamental group $\pi_{1}(\Gamma)$.
	
	\begin{defn}[{Compare with~\cite{ku-19}}]
		Let $\Gamma$ be a graph with a connection $\nabla$, equipped with a sign function $\eps$ and an invariant function $c$. Then we define the group $A(\Gamma)$ as the additive group consisting of elements of the form $f = (f_v)_{v\in V}$, $f_v\in \Z \str v$, satisfying
		\begin{equation}\label{eq:axialrel1}
			f_{t(e)}=\nabla_{e}\bigl(\eps(e)\cdot f_{i(e)}+\la f_{i(e)},e\ra\cdot c(e)\bigr),
		\end{equation}
		for any $e\in E$. 
		We call the group $A(\Gamma)$ the \textit{group of axial functions} on the graph $\Gamma$.
	\end{defn}
	
	Note that although this group is called the group of axial functions, its elements are itself not axial functions; however, we have:
	\begin{lm}\label{lm:labelaxialfunctions} 
		The groups of label and axial functions are isomorphic, via the homomorphism $\Phi:L(\Gamma)\to A(\Gamma)$ defined by
		\[
		\Phi(a)_v:=\sum_{e\in\str v} a(e) \cdot e.
		\]
		Its inverse is given by
		\[
		\Phi^{-1}(f)(e) := \langle f_{i(e)},e\rangle 
		\]
	\end{lm}
	In particular, the components $\widetilde{\alpha}_i\in L(\Gamma)$ of (the lift of) an axial function of a GKM graph (cf.\ Lemma~\ref{lem:componentsofaxialfunction}) provide, via the assignment $\Phi$, elements of $A(\Gamma)$.
	\begin{proof}
		Let us show that for $a\in L(\Gamma)$, the tuple $(\Phi(a)_v)_{v\in V}$ in the statement of the lemma defines an element of $A(\Gamma)$. For any edge $e$ we compute
		\begin{align*}
			\Phi(a)_{t(e)} &= \sum_{e'\in\str i(e)} a(\nabla_ee')\cdot \nabla_ee'\\
			&= \sum_{e'\in\str i(e)} (\langle \eps(e),e'\rangle \cdot a(e') \cdot \nabla_ee' + \langle c(e),e'\rangle \cdot a(e)\cdot \nabla_e e')\\
			&= \biggl(\sum_{e'\in\str i(e)} \langle \eps(e),e'\rangle \cdot a(e') \cdot \nabla_ee'\biggr) + a(e)\cdot \nabla_e c(e)\\
			&= \nabla_e \biggl(\eps(e)\cdot \Phi(a)_{i(e)} + \langle \Phi(a)_{i(e)},e\rangle \cdot c(e)\biggr).
		\end{align*}
		Conversely, starting with a tuple $f=(f_v)_{v\in V}\in A(\Gamma)$, we define $a$ via $a(e):= \langle f_{i(e)},e\rangle$ and check that $a\in L(\Gamma)$: 
		\begin{align*}
			a(\nabla_e e') &= \langle f_{i(\nabla_e e')},\nabla_e e') \\
			&= \langle f_{t(e)},\nabla_e e'\rangle\\
			&= \langle \eps(e),e'\rangle \cdot \langle f_{i(e)},e'\rangle + \langle f_{i(e)},e\rangle \cdot \langle c(e),e'\rangle\\
			&= \langle \eps(e),e'\rangle \cdot a(e') + \langle c(e),e'\rangle \cdot a(e).
		\end{align*}
		Clearly, these two assignments are homomorphisms and are inverse to each other.
	\end{proof}
	
	As $\Gamma$ is connected, any element $f\in A(\Gamma)$ is determined by its value $f_u$ at any vertex $u\in V$. Hence $A(\Gamma)$ may be regarded as a (free abelian) subgroup of $\Z \str u$. Let us determine those elements in $\Z \str u$ that extend (uniquely) to an element in $A(\Gamma)$. To this end consider, for any oriented edge $e\in E(\Gamma)$, the homomorphism
	\begin{equation}\label{eq:lfunc}
		\varphi_e = \varphi_{e}^{\Gamma}\colon
		\Z \str i(e) \longrightarrow \Z \str t(e);\, 
		x\longmapsto \nabla_e\bigl(\eps(e)\cdot x + \langle x,e\rangle \cdot c(e)\bigr).
	\end{equation}
	Explicitly, on generators it satisfies
	\begin{equation}\label{eq:phiongenerators} \varphi_e(e) =  \overline{e}  + \nabla_e c(e) = \overline{e} + \eps(\overline{e})\cdot c(\overline{e}) ,\qquad \varphi_e(e') = \eps(\overline{e})\cdot\nabla_e e' \quad \textrm{for} \quad e'\neq e.
	\end{equation}
	Then, a tuple $(f_v)_v$ is in $A(\Gamma)$ if and only if
	\begin{equation}\label{eq:reformulateAGamma}
		f_{t(e)} = \varphi_e (f_{i(e)}),
	\end{equation}
	holds for all oriented edges $e$.   For an edge path $\gamma = (e_1,\ldots,e_q)$ in $\Gamma$, we put $\varphi_\gamma:=\varphi_{e_q}\circ \cdots \circ \varphi_{e_1}$.
	
	\begin{lm}
		For any unsigned GKM-graph $\Gamma$, the groups $A(\Gamma)$ and $L(\Gamma)$ do not depend on the choice of lift of the axial function $\alpha$, up to a group isomorphism.
	\end{lm}
	\begin{proof}
		By Lemma~\ref{lm:labelaxialfunctions}, it is enough to prove the lemma for $A(\Gamma)$.
		Let $\widetilde{\alpha}^{i}$ be lifts of $\alpha$ with the respective $c^{i}$, $\eps^{i}$, where $i=1,2$.
		By the lift condition, one has
		\[
		\widetilde{\alpha}^{2}(e)=d(e)\cdot\widetilde{\alpha}^{1}(e),
		\]
		for some $d\colon E\to\lb\pm 1\rb$.
		From $2$-independence and the congruence relation one obtains
		\begin{equation}\label{eq:indeplift}
			\la\eps^{2}(e),e'\ra d(e')=\la\eps^{1}(e),e'\ra d(\nabla_{e}e'),\ 
			\la c^{2}(e),e'\ra d(e)=\la c^{1}(e),e'\ra d(\nabla_{e}e').
		\end{equation}
		Let
		\[
		F_{v}\colon \Z \str v\to \Z \str v,\ e\mapsto d(e)\cdot e.
		\]
		The relations \eqref{eq:indeplift} imply that the following diagram is commutative:
		\[
		\begin{tikzcd}
			\Z \str i(e) \arrow{r}{\phi^{1}_{e}} \arrow{d}{F_{i(e)}} & \Z \str t(e)\arrow{d}{F_{t(e)}}\\
			\Z \str i(e) \arrow{r}{\phi^{2}_{e}} & \Z \str t(e),
		\end{tikzcd}
		\]
		where $\phi^{i}_{e}$ is given by the formula~\eqref{eq:lfunc} with respect to $c^{i}$, $\eps^{i}$, $i=1,2$. Using \eqref{eq:reformulateAGamma}, it follows that the $F_v$ together define an isomorphism between the respective axial function groups.
	\end{proof}
	
	Any edge path defines, by linear parametrization, a curve in the topological space $\Gamma$. In particular, a cycle induces a loop, and hence a cycle based at $v$ an element in $\pi_1(\Gamma,v)$. Conversely, any element in $\pi_1(\Gamma,v)$ can be represented by a cycle based at $v$. 
	It is well known that the fundamental group of a graph is isomorphic to the quotient of the group generated by concatenations of cycles modulo equivalence relation given by elementary homotopies~\cite[\S 3.7]{sp-66}.
	
	\begin{defn}
		We call the $\pi_{1}(\Gamma,u)$-representation on $\Z \str u$ given by the formula
		\begin{equation}\label{eq:actiongen}
			[\gamma]\cdot x:=\phi_{\gamma}(x)
		\end{equation}
		the \textit{A-monodromy representation}, where $[\gamma]$ denotes the class in $\pi_{1}(\Gamma,u)$ of an edge loop $\gamma$ based at the vertex $u$ in $\Gamma$.
		The image of $\pi_{1}(\Gamma,u)$ in $\GL(\Z \str u)$ is called the \textit{A-monodromy group} at $u\in \Gamma$.
	\end{defn}
	
	\begin{pr}\label{pr:cwd}
		The A-monodromy representation is well-defined. More precisely: Let $\gamma$, $\gamma'$ be any two edge loops based at $u$ in $\Gamma$ such that $[\gamma]=[\gamma']$ holds in $\pi_{1}(\Gamma,u)$.
		Then one has $\varphi_{\gamma}=\varphi_{\gamma'}$.
	\end{pr}
	\begin{proof}
		It suffices to check that  $\phi_{\overline{e}}\cdot\phi_{e}=\Id$ for any edge $e\in E$.
		We verify this identity on the generators of $\Z \str i(e)$. 
		For any edge $e'\neq e$ at $i(e)$, 
		\[
		\varphi_{\overline{e}}\circ \varphi_e(e') = \varphi_{\overline{e}}(\eps(\overline{e})\cdot\nabla_e e') = 
		\eps(\overline{e})\cdot \eps(e)\cdot \nabla_{\overline{e}}\nabla_e e' = e'.
		\]
		Moreover, we have 
		\begin{align*}
			\varphi_{\overline{e}}(\eps(\overline{e})\cdot c(\overline{e})) &= 
			\nabla_{\overline{e}} (\eps(\overline{e})\cdot \eps(\overline{e})\cdot c(\overline{e}) + \langle \eps(\overline{e})\cdot c(\overline{e}),\overline{e}\rangle c(\overline{e}))\\
			&= \nabla_{\overline{e}} (c(\overline{e})-2c(\overline{e})) \\
			&= - \eps(e)\cdot c(e),
		\end{align*}
		by definition of sign and invariant function, whence
		\[
		\varphi_{\overline{e}}\circ \varphi_e(e) = 
		\varphi_{\overline{e}}(\overline{e} + \eps(\overline{e})\cdot c(\overline{e})) = e+\eps(e)\cdot c(e) -\eps(e)\cdot c(e) = e.
		\]
	\end{proof}
	\begin{rem}
		The $\pi_1$-actions at two different vertices $u,v$ of the connected graph $\Gamma$ are isomorphic. To see this choose any edge path $\gamma$ from $u$ to $v$ in $\Gamma$. It induces an isomorphism $\varphi_\gamma:\Z \str u\to \Z \str v$ as well as a group isomorphism $\pi_1(\Gamma,v)\to \pi_1(\Gamma,u);\, [\delta]\mapsto [\gamma^{-1}\delta\gamma]$, which intertwine the corresponding representations, as
		\[
		\varphi_\delta(\varphi_\gamma(x)) = \varphi_\gamma(\varphi_{\gamma^{-1}\delta\gamma}(x))
		\]
		for all $x\in \Z \str u$.
	\end{rem}
	\begin{pr}\label{pr:invariants}
		The group $A(\Gamma)$ of axial functions is isomorphic to $(\Z \str u)^{\pi_1(\Gamma,u)}$.
	\end{pr}
	\begin{proof}
		By~\eqref{eq:reformulateAGamma}, the value of any element in $A(\Gamma)$ at $u$ is invariant under the action of $\pi_1(\Gamma,u)$. Conversely, given an element $x\in \Z \str u$ invariant under the A-monodromy group, we define $f=(f_v)_{v\in V}$ as follows: for any $v\in V(\Gamma)$ we choose an edge path $\gamma$ from $u$ to $v$, and put $f_v:=\varphi_\gamma(x)$. By $\pi_1$-invariance, $f_v$ is independent of the choice of $\gamma$, and $(f_v)_v\in A(\Gamma)$.
	\end{proof}
	
	\begin{rem}
		Given a graph $(V,E)$ with a connection $\nabla$, such that any two edges in $(V,E)$ with a common vertex belong to a $2$-face with respect to $\nabla$.
		Consider the following problem.
		Does there exist an axial function compatible with $(V,E)$ and $\nabla$?
		The existence of an axial function is equivalent to \textit{polynomial} equations (depending on a sign function) in variables being the components $\la c(e),e'\ra$ of an invariant function $c$ at incident edges $e,e'$.
		For smooth projective toric varieties these equations appear in~\cite[pp.\ 45--46]{od-85}, and admit explicit solutions for cycles of length $3$, $4$ (ibid.).
		More generally, the equations (involving sign function values) for quasitoric manifolds were obtained in~\cite[Theorem 3]{do-01} and were solved for the edge graph of any cube (ibid.).
	\end{rem}
	
	\begin{rem}
		A collection of abelian group isomorphisms $\psi_{e}$ satisfying $\psi_{\overline{e}}\psi_{e}=\Id$ for every $e\in E$ is nothing other than a locally constant sheaf (or a coefficient system) of abelian groups on the underlying graph $G=(V, E)$ equipped with the Alexandrov topology on $V\sqcup E$.
		The topology is defined by declaring $V'\sqcup E'$ to be open for any subgraph $G'=(V', E')$ such that with every edge it contains both its vertices.
		E.g. see~\cite[\S 13]{bo-tu-82}, \cite{ba-10}.
		The group of axial functions is isomorphic to the group of global sections for the sheaf defined by the formula \eqref{eq:lfunc}.
		The group of label functions is isomorphic to the group of global sections for the sheaf defined by:
		\[
		\psi_{e}\colon (\Z\str i(e))^{*}\to (\Z\str t(e))^{*},\ 
		\psi_{e}(a):=\nabla_{e}\bigl(\eps(e)\cdot a + \langle a,e\rangle \cdot c(e)^{*}\bigr),
		\]
		where we define the action of $\mu_{2}\str v$ and $\nabla_{e}$ on the dual space as the adjoint operators with respect to the natural pairing
		\[
		\la-,-\ra\colon (\Z\str v)^{*}\times \Z\str v\to\Z,\ 
		\la a,x\ra:=a(x),
		\]
		and $c(e)^{*}$ is the dual linear function to $c(e)$.
		We have the following commutative diagram
		\[
		\begin{tikzcd}
			\Z \str i(e) \arrow{r}{\phi_{e}} \arrow{d}{\cong} & \Z \str t(e)\arrow{d}{\cong}\\
			(\Z \str i(e))^{*} \arrow{r}{\psi_{e}} & (\Z \str t(e))^{*},
		\end{tikzcd}
		\]
		where both vertical arrows are given by mapping an element to its linearly dual.
		From this perspective, Lemma~\ref{lm:labelaxialfunctions} shows that the respective sheaves are isomorphic.
	\end{rem}
	
	\subsection{Extension criterion for unsigned GKM graphs}
	
	In this section, we give a different proof of the GKM graph extension criterion~\cite{ku-19} in the slightly greater generality of unsigned GKM graphs (compared with~\cite{ku-19}) below.
	
	\begin{defn}\label{def:ext}
		Let $\Gamma$ be a signed GKM graph with the axial function $\alpha:E\to \Z^k$. A GKM graph $\Gamma'$ with the same underlying graph and the same connection as $\Gamma$ is called an \emph{extension of $\Gamma$} if its axial function $\alpha':E\to \Z^{k'}$ satisfies that there exits a group epimorphism $p\colon \Z^{k'}\to \Z^{k}$ such that 
		\begin{equation}\label{eq:extension}
			p\circ \alpha'=\alpha.
		\end{equation}
		The same definition applies to unsigned GKM graphs, as~\eqref{eq:extension} is meaningful also when $\alpha$ and $\alpha'$ take values in $\Z^k/\pm 1$ respectively $\Z^{k'}/\pm 1$.
	\end{defn}
	
	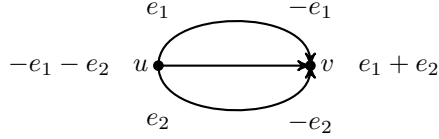
\begin{figure}[h]
		\begin{center}
			\begin{tikzpicture}[node distance=3cm, thick,main node/.style={circle,draw,font=\sffamily\Large\bfseries}]
				\fill (-1,0) coordinate (u) circle (2pt);
				\fill (1,0) coordinate (v) circle (2pt);
				\node[left] at (-1,0) {$u$};
				\node[right] at (1,0) {$v$};
				\draw[->,>=stealth'] (u)--(v);
				\draw[->,>=stealth'] (u) to [out=90,in=90] (v);
				\draw[->,>=stealth'] (u) to [out=270,in=270] (v);
				\node[above] at (-1,0.5) {$e_{1}$};
				\node[below] at (-1,-0.5) {$e_{2}$};
				\node[left] at (-1.5,0) {$-e_{1}-e_{2}$};
				\node[above] at (1,0.5) {$-e_{1}$};
				\node[below] at (1,-0.5) {$-e_{2}$};
				\node[right] at (1.5,0) {$e_{1}+e_{2}$};
			\end{tikzpicture}
			\caption{The signed GKM graph (with the connection acting by nontrivial permutation along every edge) of the canonical $T^{2}$-action on $G_{2}/SU_{3}=S^6$ does not extend to a torus graph~\cite{ku-19}; the respective unsigned GKM graph (with the different connection sending every edge to itself) extends to an unsigned torus graph, corresponding to the standard action of $T^3$ on $S^6\subset \C^3\oplus \R$
			}
			\label{fig:sph}
		\end{center}
	\end{figure}
	
	\begin{rem}
		Consider an extension $\Gamma'$ of $\Gamma$; in case the graphs are unsigned, choose the lift of $\alpha'$ as an extension of the lift of $\alpha$. Then one can show that the equalities
		\[
		\eps'=\eps,\ 
		c'=c,
		\]
		of the associated sign and invariant functions hold.
		This explains the choice of name ``invariant function'' (see~\cite{ku-19}).
	\end{rem}
	\begin{rem}
		Consider a GKM action of a torus $T'$ on a manifold $M$, and a subtorus $T$ such that the restiction of the action to $T$ is still GKM. In this situation, the GKM graph of the $T'$-action is an extension of the GKM graph of the $T$-action.
	\end{rem}
	
	\begin{thm}[Compare with~\cite{ku-19}]\label{thm:kurrk}
		For any unsigned $(n,k)$-type GKM graph $\Gamma$, $\rk A(\Gamma)\geq q$ holds if and only if there exists an unsigned $(n,q)$-GKM graph which is an extension of $\Gamma$. If $\Gamma$ is effective, the extension may be chosen to be effective as well.
		The same statement is true for signed GKM graphs.
	\end{thm}
	
	\begin{proof}
		We are given an axial function $\alpha:E\to \Z^k/\pm 1$ with lift $\widetilde{\alpha}$, with associated sign and invariant functions $\eps$ and $c$. The signed case corresponds to the constant sign function $\eps\equiv 1$. Consequently, the label function group $L(\Gamma)$ and the axial function group $A(\Gamma)$ are defined. The components $\widetilde{\alpha}_1,\ldots,\widetilde{\alpha}_k$ of $\widetilde{\alpha}$ define linearly independent elements in $L(\Gamma)$ by Lemma~\ref{lem:componentsofaxialfunction}.
		Given an $(n,q)$-type extension of $\alpha$, the same lemma implies that also its components (respectively those of its lift that extends the lift $\widetilde{\alpha}$) define $q$ linearly independent elements of $L(\Gamma)\cong A(\Gamma)$, hence $\rk A(\Gamma)\geq q$.
		
		For the converse direction,  assuming that $\rk L(\Gamma)\geq q$, we find $\widetilde{\alpha}_{k+1},\ldots,\widetilde{\alpha}_q\in L(\Gamma)$ such that $\widetilde{\alpha}_1,\ldots,\widetilde{\alpha}_q$ are linearly independent in $L(\Gamma)$. Then $\beta:=(\widetilde{\alpha}_1,\ldots,\widetilde{\alpha}_q):E\to \Z^q$  satisfies the congruence relation (see Remark~\ref{rem:labax}), and as $\beta$ is an extension of $\alpha$, the labels at each vertex are automatically pairwise linearly independent. By Lemma~\ref{lem:componentsofaxialfunction}, $\beta$ is almost effective.
		
		Let us assume additionally that $\alpha$ is effective. Denoting by $N$ the $\Z$-span of the image of $\beta$, the projection $\pi:N\to \Z^k$ onto the first $k$ components is (split, because $\Z^{k}$ is free and therefore is a projective group) surjective, by effectiveness of the axial function $\alpha$. Hence, we obtain a split short exact sequence 
		\[
		0\longrightarrow \ker \pi\longrightarrow N\longrightarrow \Z^k\longrightarrow 0
		\]
		yielding an isomorphism $N\cong \Z^k\oplus \ker \pi$; combining with an isomorphism $\ker \pi\cong \Z^{q-k}$ we obtain an isomorphism $N\cong \Z^k\oplus \Z^{q-k} = \Z^q$. We compose it with $\beta$
		\[
		E\overset\beta\longrightarrow N\longrightarrow \Z^q
		\]
		to obtain another (lift of an) axial function, whose first $k$ components are $\widetilde{\alpha}_1,\ldots,\widetilde{\alpha}_k$. By construction this axial function is effective.
	\end{proof}
	
	\begin{rem}
		If the original GKM graph $\Gamma$ is $\GKM_r$ for some $r$, i.e., at any vertex any $r$ labels are linearly independent, then any extension of $\Gamma$ is also $\GKM_r$. In general, knowledge on the rank of the group of axial functions does not give information as to whether some extension is GKM$_s$ for some $s>r$. However, if $\rk A(\Gamma)$ equals the valency $n$, then (almost) effectivity of the maximal extension implies that this extension is automatically GKM$_n$, i.e., a torus graph. 
	\end{rem}
	
	\subsection{Proof of Theorem~\ref{thm:2}}
	\begin{defn}[\cite{ta-04}]\label{lm:ptop}
		For an edge path $\gamma=( e_{1},\dots,e_{r})$ in $(V,E)$ denote its initial and terminal vertices by $i(\gamma)$, $t(\gamma)$, respectively. 
		For any $v\in V$ let $\Z \str v$ be the free $\Z$-module with generators equal to the set $\str v$. 
		A connection $\nabla$ on $(V,E)$ defines the \textit{parallel transport operator}
		\[
		\Pi_{\gamma}:\ \Z \str i(\gamma)\to \Z \str t(\gamma),\ \Pi_{\gamma}(e):=\nabla_{e_k}\circ\dots\circ\nabla_{e_1} e,\ e\in \str i(\gamma).
		\]
		In case of $i(\gamma)=t(\gamma)$, $\Pi_{\gamma}$ is called the \textit{monodromy operator} at $i(\gamma)$ along $\gamma$.
	\end{defn}
	
	\begin{thm}\label{thm:extmain}
		Let $\Gamma$ be any unsigned (signed, respectively) $\GKM_{3}$ graph such that the conjugate $2$-faces in $\Gamma$ generate the fundamental group $\pi_{1}(\Gamma)$.
		Suppose that the monodromy along any $2$-face of $\Gamma$ in the sense of Definition \ref{lm:ptop} is trivial on the transversal edges of this $2$-face. 
		Then $\Gamma$ admits an extension to an unsigned (signed, respectively) torus graph. 
	\end{thm}
	\begin{proof}
		We consider only the unsigned case, because the same argument holds when $\Gamma$ has a signed structure.
		By Proposition~\ref{pr:invariants} we have to show that the $\pi_1(\Gamma,u)$-action on $\Z \str u$ is trivial, for some base point $u$. Consider a cycle $g$ at $u$ of the form 
		\[
		h_{1}^{-1} g_{1} h_{1}\cdots h_{q}^{-1} g_{q} h_{q},
		\]
		where $g_{j}$ is a cycle bounding a $2$-face $\Gamma_j$ in $\Gamma$ and the $h_j$ are paths based at $u$. Then
		\begin{equation}\label{eq:mainthmphi}
			\varphi_{g}= \varphi_{h_{1}^{-1}} \varphi_{g_{1}} \varphi_{h_{1}}\cdots \varphi_{h_{q}^{-1}} \varphi_{g_{q}} \varphi_{h_{q}}.
		\end{equation}
		We will show that every $\varphi_{\gamma}$ acts trivially on $\Z \str v$, where $\gamma$ is a cycle bounding a $2$-face $\Gamma'$ based at some vertex  $v:=i(\gamma)$. This will conclude the proof: then all $\varphi_{g_j}$ are the identity and by \eqref{eq:mainthmphi} also $\varphi_g$, hence by the assumption on the fundamental group of $\Gamma$ the whole $A$-monodromy representation of $\pi_1(\Gamma,u)$ on $\Z \str u$ is trivial.
		
		Since the original labeling of $\Gamma$ induces on $\Gamma'$ the structure of a torus graph, $A(\Gamma)\cong (\Z \str v)^{\pi_1(\Gamma,v)}$ contains two linearly independent elements, say $x_i$, $i=1,2$. By definition of the action of $\pi_1(\Gamma,v)$, these are in particular invariant under $\varphi_{\gamma}$.
		Furthermore, $x_{i}$, $i=1,2$, remain linearly independent upon the natural projection
		\[
		\pi\colon \Z \str_{\Gamma} v\to \Z \str_{\Gamma'} v = \Z e_1 \oplus \Z \overline{e_q}.
		\]		
		If $e\in\str v$ is transversal to $\Gamma'$ (i.e., different from $e_1$ and $\overline{e_q}$), then 
		\[
		\varphi_{\gamma}(e) = \varphi_{e_q}\circ \cdots \circ \varphi_{e_1}(e) = \nabla_{e_q}\cdots \nabla_{e_1} e = e,
		\]
		as the edge $e$ stays transversal upon successive applications of the connection, and as by assumption the monodromy is trivial on transversal edges. We showed that $\varphi_{\gamma}$ acts trivially on a rank $n$ submodule of $\Z \str v$, hence trivially on all of $\Z \str v$.
	\end{proof}
	
	\begin{rem}
		The condition on the monodromy is automatically satisfied in the GKM$_4$ case.
	\end{rem}
	
	\begin{rem}
		The statement of this theorem was claimed in \cite[Theorem 2.1.3]{wa-23}; however, its proof is faulty. Indeed, the algorithm to construct a certain maximal tree in the proof of \cite[Lemma 2.1.1]{wa-23} is unclear. In the notation of \cite{wa-23}, after having chosen the connection path $\gamma_1$, it might a priori happen that every other connection path intersects $\gamma_1$ in a disconnected set. In this case, the algorithm would not produce a tree. It seems that this is related to unsolved questions on the shellability of the CW complex $Q_2$ obtained by gluing $2$-discs to every connection path of $\Gamma$.
		
		Moreover, at the end of \cite[Section 2.1]{wa-23} it is claimed that for GKM$_4$ manifolds the condition on the generation of the fundamental group is automatically satisfied. This is unclear; while by Corollary \ref{cor:perfg} the Abelianization of this fundamental group vanishes, to our knowledge there exists no argument why the fundamental group should be trivial.
	\end{rem}
	
	\begin{rem}\label{rem:genext}
		More generally, Theorem~\ref{thm:extmain} holds for any $(n,k)$-type $\GKM_{3}$ graph $\Gamma$ with countably many vertices and edges such that any $2$-face of $\Gamma$ has only finitely many distinct edges and vertices.
		This assumption is required for a cycle around an arbitrary $2$-face to be well defined.
		Notice that any cycle in $\Gamma$ has a finite number of edges by compactness.
		This shows that the above proof of Theorem~\ref{thm:extmain} carries over to such more general situation.
		We will use this generalization in Section~\ref{sec:covgkm} to prove Theorem~\ref{thm:comp1}.
	\end{rem}
	
	\section{Extensions of realizable complexity one GKM$_4$ graphs}\label{sec:covgkm}
	
	In this section we prove Theorem~\ref{thm:3} (see Theorem~\ref{thm:comp1} below) using the acyclicity result~\cite{ay-ma-23} for orbit space skeleta of GKM manifolds (cf. Theorem~\ref{thm:ams} and Corollary~\ref{cor:perfg}).
	\begin{defn}
		A morphism of graphs is called a \textit{graph covering} if it is a covering of the underlying topological spaces.
		Given GKM-graphs $\widetilde{\Gamma}$, $\Gamma$ with the underlying respective graphs $(\widetilde{V},\widetilde{E})$ and $(V,E)$, a graph covering $f\colon (\widetilde{V},\widetilde{E})\to (V,E)$ is said to be a \textit{covering of GKM graphs} if 
		\[
		\widetilde{\alpha}(e)=\alpha(f(e)),\ 
		\nabla_{f(e)}(f(e'))=f(\nabla_{e}(e')),
		\]
		hold for all edges $e,e'\in \widetilde{E}$ with $i(e)=i(e')$ .
	\end{defn}
	
	\begin{rem}
		A GKM cover is nothing other than a GKM fiber bundle~\cite{gu-sa-za-12} whose fiber has no edges.
	\end{rem}
	
	Given a $\GKM_{3}$ manifold $M$ with orbit space $Q$, let $p\colon \widetilde{Q_{2}}\to Q_{2}$ be the universal cover of the orbit space of the $2$-skeleton of the action. The underlying graph of the GKM graph $\Gamma$ is homeomorphic to $Q_1$, hence naturally a subset of $Q_2$. The preimage $\widetilde{Q_1}:=p^{-1}(Q_{1})$ then is a graph as well, and the restriction of $p$ to $p:\widetilde{Q_1}\to Q_1$ is a graph morphism.
	
	\begin{lm}\label{lm:covgkm}
		We can lift the GKM graph structure $\Gamma$ on $Q_1$ via $p$ to obtain a GKM graph $\widetilde{\Gamma}$ with (possibly infinite) underlying graph $\widetilde{Q_1}$. Consequently we obtain a GKM covering $\pi\colon \widetilde{\Gamma}\to \Gamma$. This covering is trivial over any $2$-face.
	\end{lm}
	\begin{proof}
		The map $\widetilde{\alpha}:=\alpha\circ p$ is an axial function on $\widetilde{Q_{1}}$.
		The cover $p$ over any $2$-cell $F$ of $Q_{2}$ (i.e. a $2$-face of $\Gamma$) is trivial, because $F$ is a contractible space.
		This establishes the triviality of the covering over any $2$-face.
		Therefore, we can define the connection on $\widetilde{Q_{1}}$ by lifting the connection of $\Gamma$.
		In particular, the congruence relation for $\widetilde{\Gamma}$ follows by considering any lift of any triple $e$, $e'$, $e''=\nabla_{e'} e$ of consecutive edges in any $2$-face of $\Gamma$.
		Therefore, $\widetilde{\Gamma}$ is a GKM graph.
		Clearly, the induced map $\pi\colon \widetilde{\Gamma}\to \Gamma$ is a GKM covering.
	\end{proof}
	
	Let $D:=\pi_{1}(Q_{2})$ be the deck transformation group of the covering $p$. The group $D$ acts on the graph $\widetilde{\Gamma}$ by automorphisms.
	Therefore, $D$ acts on the group of axial functions $A(\widetilde{\Gamma})$ as follows: for $f=(f_v)_{v\in V_{\widetilde{\Gamma}}}\in A(\widetilde{\Gamma})$ we put
	\begin{equation}\label{eq:dact}
		(g\cdot f)_v:= f_{g^{-1}v},\ v\in V_{\widetilde{\Gamma}},\ g\in D,\ f\in A(\widetilde{\Gamma}).
	\end{equation} 
	This action is a representation  $R\colon D\to \GL(A(\widetilde{\Gamma}))$ on the free abelian group $A(\widetilde{\Gamma})$.
	The map $\pi$ induces a well-defined pullback map $\pi^*\colon A(\Gamma)\to A(\widetilde{\Gamma})$ given by
	\begin{equation}\label{eq:pbfla}
		(\pi^{*} f)_v:=f_{\pi(v)},\ v\in V_{\widetilde{\Gamma}},\ f\in A(\Gamma).
	\end{equation}
	As $\pi$ is surjective on vertices, Equation \eqref{eq:pbfla} shows that $f$ can be reconstructed from $\pi^*f$; hence, $\pi^*$ is injective.
	Denote the subgroup of $D$-invariant elements in  $A(\widetilde{\Gamma})$ by $A(\widetilde{\Gamma})^{D}$.
	
	\begin{lm}\label{lm:pbim}
		One has the following equality of subgroups in $A(\widetilde{\Gamma})$:
		\[
		\Img \pi^{*}=A(\widetilde{\Gamma})^{D}.
		\]
	\end{lm}
	\begin{proof}
		We compare two respective subgroups of $A(\widetilde{\Gamma})$ as follows.
		The inclusion $\Img \pi^{*}\subseteq A(\widetilde{\Gamma})^{D}$ follows directly by comparing~\eqref{eq:dact} and~\eqref{eq:pbfla}.
		The group $D$ acts transitively on any fiber of $\pi$.
		Therefore, for any element $f\in A(\widetilde{\Gamma})^{D}$, the value $f_v$, $v\in \pi^{-1}(w)$, does not depend on the choice of $v$ for any fixed $w\in V_{\Gamma}$.
		This shows the inclusion $A(\widetilde{\Gamma})^{D}\subseteq \Img \pi^{*}$.
	\end{proof}
	
	The following is a simple but important observation that is used in the proof below.
	
	\begin{lm}\label{lm:rept}
		Any homomorphism of a perfect group into an abelian group is trivial.
	\end{lm}
	
	We prove Theorem~\ref{thm:3}.
	
	\begin{thm}\label{thm:comp1}
		The GKM graph of any $\GKM_{4}$ manifold of complexity one admits an extension to a torus graph.
	\end{thm}
	\begin{proof}
		Let $\Gamma$ be the GKM graph of a GKM$_4$ manifold $M$ of complexity one.
		By Lemma~\ref{lm:covgkm} we may consider the GKM covering $\pi\colon \widetilde{\Gamma}\to \Gamma$ with deck transformation group $D=\pi_1(Q_2)$. As $M$ is of complexity one, by Lemma~\ref{lem:componentsofaxialfunction} we find $n-1$ linearly independent elements in $A(\Gamma)$, which by Lemma~\ref{lm:pbim} yield $n-1$ linearly independent $D$-invariant elements $f_{1},\dots,f_{n-1}\in A(\widetilde{\Gamma})$.
		
		As by construction the fundamental group $\pi_1(\widetilde{\Gamma})$ is generated by conjugated $2$-faces, and as by the GKM$_4$ assumption the monodromy along any $2$-face is trivial on transversal edges, Theorem~\ref{thm:extmain} implies that $\widetilde{\Gamma}$ admits an extension to a torus graph. (Note that $\widetilde{\Gamma}$ is potentially an infinite graph, but see Remark ~\ref{rem:genext}. The condition of finite length for any $2$-face in $\widetilde{\Gamma}$ is satisfied by Lemma~\ref{lm:covgkm}.) So we know from Lemma~\ref{lem:componentsofaxialfunction} that $A(\widetilde{\Gamma})$ has rank $n$. We may therefore choose $f_{n}\in A(\widetilde{\Gamma})$ such that the collection $f_{1},\dots,f_{n}$ is a basis of $A(\widetilde{\Gamma})\otimes\Q$.
		
		We wish to prove that the representation $R$ of $D$ on $A(\widetilde{\Gamma})\otimes \Q$ is a trivial representation. For any $g\in D$ let 
		\[
		g\cdot f_{n}=\sum_{i=1}^{n} a_{i}(g) f_{i},\ a_{i}(g)\in\Q.
		\]
		The matrix of the linear operator $R(g)$ in the basis $f_{1},\dots,f_{n}$ differs from the identity matrix only in the $n$-th column, which consists of the entries $a_i(g)$. As the group $D$ is perfect by Corollary~\ref{cor:perfg}, the homomorphism 
		\[
		\det \circ R:D\to \Q^{\times}:=\Q\setminus \lb 0\rb,
		\] 
		into the multiplicative group $\Q^\times$ is trivial by Lemma~\ref{lm:rept}. Hence 
		\[
		a_n(g)=\det R(g)=1,\ g\in D.
		\] 
		We claim that $a_{i}\colon D\to \Q$ is a homomorphism into the additive group $\Q$ for any $i=1,\dots,n-1$.
		Indeed, for any $g,h\in D$ one has
		\[
		\sum_{i=1}^n a_i(gh)f_i=g\cdot (h\cdot f_{n})=
		\sum_{i=1}^{n-1} a_{i}(h) g\cdot f_{i} + g\cdot  f_{n}=
		\sum_{i=1}^{n-1} a_{i}(h) f_{i} + \sum_{i=1}^n a_i(g)f_i.
		\]
		Here we used the $D$-invariance property of $f_{1},\dots,f_{n-1}$.
		Then Lemma~\ref{lm:rept} implies that $a_{i}=0$ for all $i=1,\dots,n-1$.
		Therefore, $f_{n}$ is $D$-invariant, and the representation $R$ is trivial.
		By Lemma~\ref{lm:pbim} we conclude that $\rk A(\Gamma)=n$.
		Now the claim on the extension to a torus graph for $\Gamma$ follows directly from Theorem~\ref{thm:kurrk}.
	\end{proof}
	
	\begin{rem}
		Consider a more general situation (related to Problem~\ref{prob:mas} $(ii)$), when $\Gamma$ is any (finite) abstract $\GKM_{4}$ graph.
		Then there exists a GKM cover $\pi\colon \widetilde{\Gamma}\to \Gamma$ that is trivial over $2$-faces such that its fundamental group is generated by conjugated $2$-faces. 
		Then $\widetilde{\Gamma}$ extends to a torus graph by Theorem~\ref{thm:extmain}.
		To prove this claim glue in $2$-discs to every $2$-face in $\Gamma$ and consider the universal covering of the resulting space as described above, as well as $\widetilde{\Gamma}:=\pi^{-1}(\Gamma)$. In this situation, however, we cannot apply  Corollary~\ref{cor:perfg}.
	\end{rem}
	
	\begin{rem}
		Notice that Theorem~\ref{thm:3} allows to drop the balanced (i.e., bipartite) condition in~\cite[Theorem 1.2]{ay-ma-so-23}, as conjectured in~\cite[Remark 5.22]{ay-ma-so-23}.
		This condition was required for the proof that $\pi_{1}(\Gamma)$ is generated by $2$-faces, and our argument does not use this condition.
		Furthermore, we point out the gap in the proof of~\cite[Proposition 5.18]{ay-ma-so-23} which our argument allows to bypass. 
		Namely, it is not clear why the orbit filtration $Q_{n-1}\to Q$ induces monomorphisms of the respective fundamental groups.
	\end{rem}

\end{document}